\documentclass[a4paper,10pt]{article}

\usepackage[utf8]{inputenc}
\usepackage[T1]{fontenc}
\usepackage[dvips]{graphicx}
\usepackage{amssymb}
\usepackage{amsmath}
\usepackage{amsfonts}
\usepackage{amsthm}

\usepackage[top=2cm,bottom=2cm,left=2cm,right=2cm]{geometry}

\newtheorem{theorem}{Theorem}[section]
\newtheorem{prop}[theorem]{Proposition}
\newtheorem{cor}[theorem]{Corollary}

\newtheorem{lem}[theorem]{Lemma}
\newtheorem{dfn}[theorem]{Definition}

\numberwithin{equation}{section}

\def\RR{\mathbb{R}}

\def\Rn{\mathbb{R}^{n}}

\def\Hndos{\mathcal{H}^{n-2}}
\def\Hntres{\mathcal{H}^{n-3}}

\DeclareMathOperator{\Img}{Image}

\begin{document}
\title{Cut and singular loci up to codimension 3}

\author{Pablo Angulo Ardoy\\
Department of Mathematics\\Universidad Autónoma de Madrid. \and 
Luis Guijarro \\
Department of Mathematics\\Universidad Autónoma de Madrid.\\ICMAT CSIC-UAM-UCM-UC3M}

\date{31 August 2009}

\maketitle

\begin{abstract}
We give a new and detailed description of the structure of cut loci, with direct applications to the singular sets of some Hamilton-Jacobi equations. These sets may be non-triangulable, but a local description at all points except for a set of Hausdorff dimension $n-2$ is well known. We go further in this direction by giving a clasification of all points up to a set of Hausdorff dimension $n-3$.
\end{abstract}

\section{Introduction}

In this paper we improve the current knowledge about the sets known as the \emph{cut locus} in differential geometry and the \emph{singular set} of solutions to static Hamilton-Jacobi equations:
\begin{eqnarray}
H(x,du(x))=1&\quad&x\in M\label{HJequation}\\
u(x)=g(x)&&x\in \partial M\label{HJboundarydata}
\end{eqnarray} 
for $H$ smooth and convex in the second argument and $g$ satisfying a standard \emph{compatibility condition} (see \ref{compatibility condition}).

The solution to the equations above is given by the Lax-Oleinik formula:
\begin{equation}\label{Lax-Oleinik} 
u(p)=\inf_{q\in \partial M}
\left\lbrace 
   d(p,q)+ g(q)
\right\rbrace
\end{equation} 
where $d$ is the distance function of a Finsler metric constructed in $\Omega$ from the hamiltonian function $H$.
Thus, when $g=0$, the solution to the equations is the distance to the boundary, and then the singular set of the solution is the cut locus from the boundary (see \cite{Li Nirenberg}), an object of differential geometry. In section \ref{section: Hamilton-Jacobi} we find a similar relationship when $g\neq 0$.

Our main result is a local description around any point of the cut locus \emph{except for a set of Hausdorff dimension $n-3$} (see Theorem \ref{complete description}).

This structure result was originally motivated by its use in the paper \cite{Nosotros2}.
This application motivated some important decisions. For example, all the proofs apply to the more general \emph{balanced split locus}. 
We show in this paper that cut loci (hence singular sets of solutions to HJ equations) are balanced split loci.
In general, there are many balanced split loci besides the cut locus. In the paper \cite{Nosotros2}, and using the results in this paper, we study and classify all possible balanced split loci.

We believe that our description of the cut locus could also be useful in other contexts. For instance, the study of brownian motion on manifolds is often studied on the complement of the cut locus from a point, and then the results have to be adapted to take care of the situation when the brownian motion hits the cut locus. As brownian motion almost never hits a set with null $\Hndos$ measure, we think our result might be useful in that field.


The paper is divided in six sections besides this introduction and an appendix.
For the convenience of the reader we have included separate statements of our results in section \ref{Section: results}  together with examples showing that some of them are sharp, a compendium of previous results in the literature and suggestions for future work.  In section \ref{section: Hamilton-Jacobi} we enlarge the class of Hamilton-Jacobi problems for which our results apply: this allows to expand the applicability of a result by Li and Nirenberg (cf. \cite{Li Nirenberg}).
Section \ref{section: definitions, split and balanced} contains all the necessary definitions that we use along the paper; although some of them have already appeared  elsewhere,  we have considered useful to collect them here in order to save the reader some effort. More important, this section contains also the key notions of 
\emph{split locus} and \emph{balanced split locus}, that play a key role in the rest of the paper.  
In section \ref{section: Balanced} we show that the cut locus of a submanifold in a Finsler metric is a balanced set. This is an extension of the corresponding Riemannian claim in \cite{Itoh Tanaka 2}, and it is necessary in order to apply our results in situations requiring the extra Finsler generality, as for instance in the already mentioned Hamilton-Jacobi problems.
Section \ref{section: First} proves our results concerning focal vectors in a balanced split locus (in the context of  a cut locus, focal minimizing geodesics), and section \ref{section: Second} contains the results about the structure of balanced split loci up to codimension 3.
An Appendix contains some important facts about Finsler exponential maps.

\paragraph{Acknowlegdements}
The first author came upon this problem after working with Yanyan Li, who gave many insights. The authors benefited from conversations with Luc Nguyen and Juan Carlos \'Alvarez Paiva.
Both authors were partially  supported during the preparation of this work by grants MTM2007-61982 and MTM2008-02686 of the MEC and the MCINN respectively.

\section{Statements of results}\label{Section: results}
\subsection{Setting}\label{section: setting}

From now on, we will work in the following setting:

\begin{itemize}
 \item A $C^{\infty}$ Finsler manifold $M$ with \emph{compact} boundary $\partial M$. The space $M\cup \partial M$ need not be compact.
 \item  The geodesic vector field $r$ in $TM$.
 \item A smooth map $M:\partial M \rightarrow TM$  that is a section of the projection map $\pi:TM\rightarrow M$ of the tangent to $M$, and such that $\Gamma(x)$ points to the inside of $M$ for every $x\in \partial M$.
\end{itemize}

Let $\Phi$ be the flow of $r$, and $D(\Phi)$ its domain. We introduce the set $V$:
\begin{equation}\label{V is ...}
 V=\left\lbrace \Phi(t,\Gamma(x)), t\geq 0, x\in \partial M, (t,\Gamma(x))\in D(\Phi) \right\rbrace
\end{equation}
The interior of $V$ is locally invariant under $\Phi_{t}$ (equivalently, $r$ is tangent to $V$).
We set $F$ to be the map $\pi\vert_{V}:V\rightarrow M$. 

We say a point $x\in V$ is a \emph{focal} point iff $d_{x}F$ is a singular map, and call $\dim \ker(d_{x}F)$ the \emph{order} of $x$. Finally, let $S$ be a balanced split locus for this setting.

\paragraph{Remark.}
Our results covers both the cut locus from a point and the cut locus from a hypersurface.
However, let us recall that, when the interest is in the cut locus, we only need to consider the exponential map from an hypersurface. The cut locus of a point is also the cut locus of a small sphere centered at the point.
In this way, our \emph{focal points} with respect to the sphere are the \emph{conjugate points} with respect to the point.
The cut locus of a smooth submanifold is also the cut locus of an $\varepsilon$-neighborhood of the submanifold.


Observe also that some authors use the term conjugate instead of focal, even when studying the distance function from a hypersurface (see for instance \cite{Li Nirenberg}).
\subsection{Results}

We will show that a cut locus is a balanced split locus (see section \ref{section: definitions, split and balanced} for the definition of this term and section \ref{section: Balanced} for the proof), so the reader may simply think that the following results apply to the cut locus. In this situation, the set $R_{p}$ with $p\in M$ consists of the
vectors tangent to the minimizing geodesics from $p$ to $\partial M$.
Nonetheless, the notation for the general case is explained in definition \ref{the set R_p}.

Our main result asserts that we can avoid focal points of order $2$ and above if we neglect a set of Hausdorff dimension $n-3$.:

\begin{theorem}[Focal points of order $2$]\label{main theorem 3}
There is a set $N\subset S$ of Hausdorff dimension at most $n-3$ such that for any $p\in S\setminus N$  and $x\in V$ such that $F(x)=p$ and $d_x F(r_{x})\in R_{p}$:
$$
\dim (\ker d_{x}F)\leq 1
$$
\end{theorem}

Combining this new result with previous ones in the literature, we are able to provide the following description of a cut locus. All the extra results required for the proof of these result will be proved in this paper, for the convenience of the reader, and also because some of them had to be slightly generalized to serve our purposes.

\begin{theorem}[The cut locus up to $\mathcal{H}$-codimension 3] \label{complete description}
Let $S$ be either the cut locus of a point or submanifold in a Finsler manifold or the closure of the singular locus of a solution of \ref{HJequation} and \ref{HJboundarydata}.
Then $S$ consists of the following types of points :

 \begin{itemize}
 \item\textbf{Cleave points:} Points at which $R_{p}$ consists of two non-focal vectors. The set of cleave points is a smooth hypersurface;
 \item \textbf{Edge points:} Points at which $R_{p}$ consists of exactly one vector of order 1. This is a set of Hausdorff dimension at most $n-2$;
 \item\textbf{Degenerate cleave points:} Points at which $R_{p}$ consists of two vectors, such that one of them is conjugate of order 1, and the other may be non-conjugate or conjugate of order 1. This is a set of Hausdorff dimension at most $n-2$;
 \item\textbf{Crossing points:} Points at which $R_{p}$ consists of non-focal and focal vectors of order 1, and $R^{\ast}_{p} $ is contained in an affine subspace of dimension $2$. This is a rectifiable set of dimension at most $n-2$;
 \item \textbf{Remainder:} A set of Hausdorff dimension at most $n-3$;
\end{itemize}

\end{theorem}

Finally, in regard to singular sets of viscosity solutions to HJ equations, we prove the following  extension Theorem 1.1 of \cite{Li Nirenberg}. In this result $\partial M$ may not be compact.

\begin{theorem}\label{Hamilton Jacobi}
Let $S$ be the singular set of a solution to the Hamilton-Jacobi system 
\begin{eqnarray*}
H(x,du(x))=1&\quad&x\in M\\
u(x)=g(x)&&x\in \partial M
\end{eqnarray*} 
where $g:\partial M\rightarrow \RR$ is a positive smooth function such that
$\left\vert g(y)-g(z)\right\vert < kd(y,z)$
for some $k<1$.
If $\mu$ is the function whose value at $y\in \partial M $ is the distance to $S$ along the unique characteristic departing from $y$, then 
\begin{enumerate}
\item $\mu$ is Lipschitz.
\item If in addition $\partial M $ is compact,
then the $(n-1)$-dimensional Hausdorff measure of $S\cap K$ is finite for any compact $K$.
\item $S$ is a Finsler cut locus from the boundary of some Finsler manifold. 
\end{enumerate}
\end{theorem}

\subsection{Examples}

We provide examples of Riemannian manifolds and exponential maps which illustrate our results.

First, consider a solid ellipsoid with two equal semiaxis and a third larger one. This is a 3D manifold with boundary, and the geodesics  starting at the two points that lie further away from the center have a first focal of order $2$ while remaining  minimizing up to that point. 
This example shows that our bound on the Hausdorff dimension of the points in the cut locus with a minimizing geodesic of order $2$ cannot be improved.

Second, consider the surface of an ellipsoid with three different semiaxis (or any generic surface as in \cite{Buchner}, with metric close to the standard sphere) and an arbitrary point in it. It is known that in the tangent space the set of first focal points is a closed curve $C$ bounding the origin, and at most of these points the kernel of the exponential map is transversal to the curve $C$. More explicitely, the set $C^{\ast}$ of points of $C$ where it is not transversal is finite. Consider then the product $M$ of two such ellipsoids. The exponential map onto $M$ has a focal point of order $2$ at any point in $(C\setminus C^{\ast})\times (C\setminus C^{\ast})$, and the kernel of the exponential map is transversal to the tangent to $C\times C$. Thus the image of the set of focal points of order $2 $ is a smooth manifold of codimension $2$. 

This example shows that the statement of theorem \ref{main theorem 3} cannot be simplified to say only that the image of the focal points of order $2$ has Hausdorff dimension $n-3$.

Finally, recall the construction in \cite{Gluck Singer}, where the authors build a riemannian surface whose cut locus is not triangulable. Their example shows that the set of points with a focal minimizing geodesic can have infinite $\Hndos$ measure. A similar construction replacing the circle in their construction with a 3d ball shows that the set of points with a minimizing geodesic focal of order $2$ can have infinite $\Hntres$ measure.

\subsection{Relation to previous results in the literature}\label{subsection: Relation to previous results in the literature}

Our structure theorem generalizes a standard result that has been proven several times by mathematicians from different fields (see for example \cite{Barden Le}, \cite{Hebda}, \cite{Mantegazza Mennucci} and \cite{Itoh Tanaka 98}):

\begin{quote}
A cut locus in a Riemannian manifold is the union of a smooth $(n-1)$-dimensional manifold $\mathcal{C}$ and a set of zero $(n-1)$-dimensional Hausdorff measure (actually, a set of Hausdorff dimension at most $n-2$). 
The set $\mathcal{C}$ consists of cleave points, which are joined to the origin or initial submanifold by exactly two minimizing geodesics, both of which are non-focal.
\end{quote}

We observe that this theorem follows from our theorem \ref{complete description}, since the union of edge, degenerate cleave, and crossing points is a set of Hausdorff dimension at most $n-2$. Our main contribution is to show that, up to codimension 3, these latter ones are the only new type of points that can appear.

The statement on cleave points quoted above follows from lemmas \ref{main theorem 2}, \ref{main theorem 4}, and \ref{theorem: conjugate of order 1} only. 
Theorem \ref{main theorem 3} is not necessary if a description is needed only up to codimension $2$.
The proof of the three lemmas is simple and has many features in common with earlier results on the cut locus. However, we have decided to include a proof of them that applies to balanced split loci, because not every balanced split locus coincides with the cut locus (see \cite{Nosotros2}), and the extra generality is necessary for forthcoming work.

In a previous paper, A. C. Mennucci studied the singular set of solutions to the HJ equations with only $C^k$ regularity. 
Under this hypothesis, the set $S\setminus \mathcal{C}$ may have Hausdorff dimension strictly between $n-1$ and $n-2$ (see \cite{Mennucci}).
We work only in a $C^{\infty}$ setting, and under this stronger condition, the set $S\setminus \mathcal{C}$ has always Haussdorf dimension at most $n-2$.

Our result \ref{complete description} uses the theory of singularities of semi-concave functions that can be found for example in \cite{Alberti Ambrosio Cannarsa}.
Though their result can be applied to a Finsler manifold, we had to give a new proof that applies to balanced sets instead of just the cut locus.

Finally, the very definition of \emph{balanced split locus} is inspired in lemma 2.1 of \cite{Itoh Tanaka 2}. Slight changes were required to adapt the property to Finsler manifolds, and the proof of the lemma itself.

\subsection{Further questions}

Theorem \ref{main theorem 3} and the classical result quoted earlier suggest the following conjecture: although the image of the focal points of order $k$ in an exponential map can have Hausdorff dimension $n-k$, the set of points in $M$ with a \emph{minimizing} geodesic of order $k$ only has Hausdorff dimension $n-k-1$.

The examples in the above section can be extended to focal points of greater order without pain, showing that this conjecture cannot be improved.

In this paper all the structure results about cut loci follow from the split and balanced properties of a cut locus. We will address the question of how many balanced split sets are there in a future paper. We believe this approach is an interesting way to look at viscosity solutions and their relation with classical solutions by characteristics.

Finally, we would like to mention that similar hypothesis and similar structure results hold in other settings. It would be interesting to study the structure of the singular locus of the solutions to other Hamilton-Jacobi equations, when the Hamiltonian depends not only on $x$ and $du$, but also in $t$ and $u$ itself, for the Dirichlet and Cauchy problems, or maybe without the convexity hypothesis on $H$.

\section{Singular locus of Hamilton-Jacobi equations}\label{section: Hamilton-Jacobi}

In this section we study the relationship between Hamilton-Jacobi equations and Finsler geometry.
The reader can find more details in \cite{Li Nirenberg} and \cite{Lions}. 

Let $M$ be an open set (or manifold) $M$ with possibly non-compact boundary.
We are interested on  solutions to the system

 \begin{eqnarray}
 H(x,du(x))=1&\quad&x\in M\notag\\
 u(x)=g(x)&&x\in \partial M\notag
 \end{eqnarray}

\noindent where  $H:T^{\ast}M\rightarrow \RR$ is a smooth function that is $1$-homogeneous and subadditive for linear combinations of covectors lying over the same point $p$, and $g:\partial M\rightarrow \RR$ is a smooth function that satisfies the following compatibility condition:
\begin{equation}\label{compatibility condition}
\left\vert g(y)-g(z)\right\vert < kd(y,z)\quad \forall y, z\in \partial  M 
\end{equation}
for some $k<1$.

As is well known, \emph{the unique viscosity solution} is given   by the Lax-Oleinik formula:
\begin{equation} 
u(p)=\inf_{q\in \partial M}
\left\lbrace 
   d(p,q)+ g(q)
\right\rbrace
\end{equation} 
where $d$ is the distance induced by the Finsler metric that is the pointwise dual of the metric in $T^{\ast}M$ given by $H$:
\begin{equation}\label{phi is the dual of H}
 \varphi_{p}(v)=\sup\left\lbrace 
\left\langle v,\alpha\right\rangle_{p}\, :\,
\alpha\in T^{\ast}_{p}M, \,H(p,\alpha)=1
\right\rbrace 
\end{equation} 
A local \emph{classical} solution can be computed near $\partial M$ following \emph{characteristic} curves, which are geodesics of the metric $\varphi$ starting from a point in $\partial M$ with initial speed given by a vector field on $\partial M$ that we call the characteristic vector field. The viscosity solution can be thought of as a way to extend the classical solution to the whole $M$.

When $g=0$, the solution (\ref{Lax-Oleinik}) is the \emph{distance to the boundary}. It can be found in \cite{Li Nirenberg}, among others, that the closure of the singular set of this function is the \emph{cut locus}\/, given for example by:
\begin{equation}\label{characterization of cut locus}
 S=\left\lbrace x\in M:
\begin{tabular}{l}
 there are at least two minimizing geodesics from $\partial M$ to $x$\\
 or the unique minimizing geodesic is focal
\end{tabular} 
\right\rbrace 
\end{equation}

Hamilton-Jacobi equations fit our setting if we let the vector field $r$ be the geodesic vector field, and $\Gamma$ be the vector field at $\partial M$ that is tangent to the departing characteristics. The map $F:V\rightarrow M$ is the map sending $(x,t)\in \partial M\times \RR$ to $ \gamma_{v(x)}(t)$, for the geodesic $\gamma$ with initial speed $v(x)$, 
where $v:\partial M \rightarrow SM$ is the characteristic vector field, and $V\subset \partial M\times \RR$ is the domain of definition of $F$. The characteristic vector at $x$ is the inner pointing normal if $g=0$ (see the appendix for the definition of normal under Finsler conditions).

Our intention in this section is to adapt this result to the case $g>0$. If $\partial M$ is compact, a global constant can be added to an arbitrary $g$ so that this is satisfied and $S$ is unchanged.
We still require that $g$ satisfies the compatibility condition \ref{compatibility condition}.
Under these conditions, our strategy will be to show that the Finsler manifold $(M,\varphi)$ can be embedded in a new manifold with boundary $(N,\tilde{\varphi})$ such that $u$ is the restriction of the unique solution $\tilde{u}$ to the problem
\begin{eqnarray}
\tilde{H}(x,d\tilde{u}(x))=1&\quad&x\in N\notag\\
\tilde{u}(x)=0 &&x\in \partial N\notag
\end{eqnarray}
thus reducing to the original problem ($\tilde{H} $ and $\tilde{\varphi}$ are dual to one another as in \ref{phi is the dual of H}). This allows us to characterize the singular set of  (\ref{Lax-Oleinik}) as a cut locus, as well as draw  conclusions  similar to those in \cite{Li Nirenberg}.

\begin{dfn}\label{indicatrix}
The \emph{indicatrix} of a Finsler metric $\varphi$ at the point $p$ is the set
$$I_{p}=\left\lbrace v\in T_{p}M\,: \, \varphi(p,v)=1\right\rbrace  $$
\end{dfn}

\begin{lem}\label{cambia metrica respeta campo} 
Let $\varphi_{0}$ and $\varphi_{1}$ be two Finsler metrics in an open set $U$, and let $X$ be a vector field in $U$ such that:
\begin{itemize}
\item The integral curves of $X$ are geodesics for $\varphi_{0}$.
\item $\varphi_{0}(p,X_{p})=\varphi_{1}(p,X_{p})=1$
\item At every $p\in U$, the tangent hyperplanes to the indicatrices of $\varphi_{0}$ and $\varphi_{1}$ in $T_{p}U$ coincide.
\end{itemize}
Then the integral curves of $X$ are also geodesics for $\varphi_{1}$
\end{lem}

 \begin{proof}
Let $p$ be a point in $U$. Take bundle coordinates of $T_{p}U$ around $p$ such that $X$ is one of the vertical coordinate vectors. An integral curve $\alpha$ of $X$ satifies:
$$(\varphi_{0})_{p}(\alpha(t),\alpha'(t))=
  (\varphi_{1})_{p}(\alpha(t),\alpha'(t))=0$$
because of the second hypothesis. The third hypothesis imply:
$$ (\varphi_{0})_{v}(\alpha(t),\alpha'(t)) =
  (\varphi_{1})_{v}(\alpha(t),\alpha'(t))$$
So inspection of the geodesic equation:
\begin{equation}
\varphi_{p}(\alpha(t),\alpha'(t))=
\frac{d}{dt} \left( \varphi_{v}(\alpha(t),\alpha'(t))\right) 
\end{equation}
shows that $\alpha$ is a geodesic for $\varphi_{1}$. 
\end{proof}

\begin{cor}\label{una Riemann con las geodesicas de una Fisnler} 
Let $\varphi$ be a Finsler metric and $X$ a vector field whose integral curves are geodesics. Then there is a Riemannian metric for which those curves are also geodesics.
\end{cor}
\begin{proof} 
The Riemannian metric $g_{ij}(p)=\frac{\partial}{\partial v_{i}v_{j}}\varphi (p,X)$ is related to $\varphi $ as in the preceeding lemma. 
\end{proof}

\begin{lem}\label{characterization of Finsler geodesics} 
Let $X$ be a non-zero geodesic vector field in a Finsler manifold and $\omega$ its dual differential one-form. Then the integral curves of $X$ are geodesics if and only if the Lie derivative of $\omega$ in the direction of $X$ vanishes.
\end{lem}

\begin{proof}
Use lemma \ref{una Riemann con las geodesicas de una Fisnler} to replace the Finsler metric with a Riemann metric for which $\omega$ is the standard dual one-form of $X$ in Riemannian geometry. Now the lemma is standard.
\end{proof}

\begin{prop}
Let $ M $ be an open  manifold with smooth boundary and a Finsler metric $\varphi$. Let $X$ be a smooth transversal vector field in 
$\partial M $ pointing inwards (resp. outwards). 
Then $ M $ is contained in a larger open manifold admitting a smooth extension $\tilde{\varphi}$ of $\varphi$
to this open set such that the geodesics starting at points $p\in \partial M $ with initial vectors $X_{p}$
can be continued indefinitely backward (resp. forward) without intersecting each other.
\end{prop}

\begin{proof} 
We will only complete the proof for a compact open set $ M $ and inward pointing vector $X$,
as the other cases require only minor modifications.

We start with a naive extension $\varphi'$ of $\varphi$ to a larger open set $ M _{2} \supset  M $. The geodesics with initial speed $X$ can be continued backwards to $ M _{2}$, and 
there is a small $\varepsilon$ for which the geodesics starting at $\partial M $ do not intersect each other for negative values of time before the parameter reaches $-\varepsilon$.

Define

$$
P:\partial  M \times (-\varepsilon,0]\to M_2, \qquad P(q,t):= \alpha_{q}(t)
$$

\noindent where $\alpha_{q}:(-\varepsilon,0]\to M_2$ is the geodesic of $\varphi'$ starting at the point $q\in \partial  M $ with initial vector $X_{q}$. 
When $p\in U_{\varepsilon}:=\Img(P)$ there is a unique  value of $t$ such that $p=P(q,t)$ for some $q\in \partial M$. We will denote such $t$ by $d(p)$.  
Extend also  the vector $X$ to $U_{\varepsilon}$ as
$X_{p}= \dot{\alpha_{q}}(t)$
where $p=P(q,t)$.

Let $c: (-\varepsilon,0]\rightarrow [0,1]$
be a smooth function such that 

\begin{itemize}
\item $c$ is non-decreasing
\item $c(t)=1 \text{ for } -\varepsilon/3 \leq t$
\item $c(t)=0 \text{ for } t\leq -2\varepsilon/3$
\end{itemize}
and finally define
$$\tilde{X}_{p}=c(d(p))X_{p}$$
in the set $U_{\varepsilon}$.

Let $\omega_{0}$ be the dual one form of $\tilde{X}$ with respect to $\varphi$ for points in $\partial M $, and let $\omega$ be the one form in $U_{\varepsilon}$ whose Lie derivative in the direction $\tilde{X}$ is zero and which coincides with $\omega_{0}$ in $\partial M $.
Then we take any metric $\varphi''$ in $U_{\varepsilon}$ (which can be chosen Riemannian) such that $\tilde{X}$ has unit norm and the kernel of $\omega$ is tangent to the indicatrix at $\tilde{X}$.

By lemma \ref{characterization of Finsler geodesics}, the integral curves of $\tilde{X}$ are geodesics for $\varphi''$. Now let $\rho$ be a smooth function in $U_{\varepsilon}\cup  M $ such that $\rho\vert_{ M }=1$, $\rho\vert_{U_{\varepsilon} \setminus U_{\varepsilon/3}}=0$ and $0\leq \rho \leq 1$, and define the metric:
$$\tilde{\varphi}= \rho(p)\varphi(p,v)+ (1-\rho(p))\varphi''(p,v)$$

This metric extends $\varphi$ to the open set $U_{\varepsilon}$ and makes the integral curves of $\tilde{X}$ geodesics.
As the integral curves of $X$ do not intersect for small $t$, 
the integral curves of $\tilde{X}$ reach infinite length before they approach $\partial U_{\varepsilon}$ and the last part of the statement follows.
\end{proof}

Application of this proposition to $ M $ and the characteristic, inwards-pointing vector field $v$ yields a new manifold $N$ containing $ M $, and a metric for $N$ that extends $\varphi$ (so we keep the same letter) such that the geodesics departing from $\partial M $ which correspond to the characteristic curves continue indefinitely backwards without intersecting.

This allows the definition of 
$$
\tilde{P}:\partial  M \times (-\infty,0]\to N, \qquad P(q,t):= \tilde{\alpha}_{q}(t)
$$
where $\tilde{\alpha} $ are the geodesics with initial condition $X$, continued backwards. Finally, define $\tilde{u}:U \rightarrow \RR$ by:
\begin{equation}\label{def of utilde}
\tilde{u}(x)=\begin{cases}
    g(y)+t& x=\tilde{P}(y,t), \quad x\in N\setminus  M \\
    u(x)& x\in  M 
     \end{cases}
\end{equation}
We notice that both definitions agree in an \emph{inner} neighborhood of $\partial M $, so the function $\tilde{u}$ is a smooth extension of $u$ to $N$.

\begin{theorem}\label{reduce HJ to boundary value 0}
Let $\Lambda=\tilde{u}^{-1}(0)$.
Then the following identity holds in $\left\lbrace \tilde{u}\geq 0 \right\rbrace $
:
\begin{equation}
\tilde{u}(x)= d(x,\Lambda)
\end{equation} 
\end{theorem}

\begin{proof} 

Let $g_{t}$ be the flux associated to the characteristic vector field $X$. By definition of $\tilde{u}$, we see that:
$$g_{t}^{\ast}\tilde{u}(x)=\tilde{u}(x)+t $$
at least for $(x,t)$ in an open set $\mathcal{O}$ containing $N\setminus  M \times (-\infty,0]$.
We deduce that $g_{t}$, restricted to a small ball $B$, sends the intersection of a level set of $\tilde{u}$ with the ball to another level set of $\tilde{u}$, whenever $t$ is small enough so that $g_{t}(B)$ is contained in $\mathcal{O}$.

In particular, the tangent distribution to the level sets is transported to itself by the flow of $X$.
On the other hand, the orthogonal distribution to $X$ is also parallel, so if we show that they coincide near $\partial M $, we will learn that they coincide in $\mathcal{O}$.

Now recall that inside $ M $, $\tilde{u}$ coincides with $u$, which is also given by the Lax-Oleinik formula \ref{Lax-Oleinik}. 
Let $y\in \partial M $ and $t>0$ small.
This formula yields the same value as the local solution by characteristics, and
we learn that the point $y$ is the closest point to $g_{t}(y)$ on the level curve $\left\lbrace u=u(y)\right\rbrace $. By appeal to lemma 2.3 in \cite{Li Nirenberg}, or reduction to the Riemannian case as in \ref{una Riemann con las geodesicas de una Fisnler}, we see that the level set $\left\lbrace u=u_{0}\right\rbrace $ is orthogonal to the vector $X_{y}$.
It follows that, in $\mathcal{O}$:
$$
H(x, d\tilde{u}(x))=\sup \lbrace\, d\tilde{u}(x)(Y)\,:\, \varphi(Y)=1\,\rbrace =d\tilde{u}(x)(\tilde{X})=1
$$

In order to show that $\tilde{u} $ and $d(\,\cdot\,,\Lambda) $ agree in $U$, we use the uniqueness properties of viscosity solutions.
Let $ N$ be the open set where $\tilde{u}>0$.
The distance function to $\Lambda$ is characterized as the unique viscosity solution to:

\begin{itemize}
\item $\tilde{u}=0$ in $\Lambda$
\item $H(x, d\tilde{u}(x))=1$ in $N$
\end{itemize}

Clearly $\tilde{u}$ satisfies the first condition. It also satisfies the second for points in the set $ M $ because it coincides with $u$, and for points in $ N \setminus  M $ because $H(x, d\tilde{u}(x))=1$ there. 

\end{proof}

\begin{proof}[Proof of Theorem \ref{Hamilton Jacobi}] 
The first part follows immediately from Theorem \ref{reduce HJ to boundary value 0} and Theorem 1.1  in \cite{Li Nirenberg}.
The second is an easy consequence of the first, while the last is contained in the results of this section.  
\end{proof}

\paragraph{Remark.} Regularity hypothesis can be softened. 
In order to apply the results in \cite{Li Nirenberg}, it is enough that the geodesic flow, the characteristic vector field and $g$ itself are $C^{2,1}$, which implies that $\Lambda$ is $C^{2,1}$. Thus the result in true for less regular hamiltonians and open sets.

\section{Split locus and balanced split locus} \label{section: definitions, split and balanced}

We now introduce some properties of a set necessary in the proofs of our results. We prove in section \ref{section: Balanced} that a cut loci in  Finsler manifolds have all of them.

\begin{dfn}\label{vector de x a y}
For a pair of points $p, q\in  M $ such that $q$ belongs to a convex neighborhood of $p$, we define, following \cite{Itoh Tanaka 2}, 
\begin{equation}
 v_{p}(q)=\dot{\gamma}(0)
\end{equation} 
as the speed at $0$ of the unique unit speed minimizing geodesic $\gamma$ from $p$ to $q$.
\end{dfn}

\begin{dfn}\label{approximate tangent cone}
The \emph{approximate tangent cone} to a subset $E$ at $p$ is:
$$T(E,p)=\left\lbrace
r\theta:\theta=\lim v_{p}(p_{n}),
\exists \lbrace p_{n}\rbrace\subset E, p_{n}\rightarrow p, r>0 
\right\rbrace$$
and the \emph{approximate tangent space} $Tan(E,p)$ to $E$ at $p$ is the vector space generated by $T(E,p)$.
\end{dfn}

We remark that the definition is independent of the Finsler metric, despite its apparent dependence on the vectors $v_{p}(p_{n}) $. 

\begin{dfn}\label{splits}
For a set $S\subset M$, let $A(S)$ be the union of all integral segments of $r$ with initial point in $\Gamma$ whose projections in $M$ do not meet $S$.
We say that a set $S\subset M$ \emph{splits} $M$ iff $\pi$ restricts to a bijection between $A(S)$ and $M\setminus S$.
\end{dfn}

Whenever $S$ splits $M$, we can define a vector field $R_{p}$ in $M\setminus S$ to be $dF_{x}(r_{x})$ for the unique $x$ in $V$ such that $F(x)=p$ and there is an integral segment of $r$ with initial point in $\Gamma$ and end point in $x$ that does not meet $F^{-1}(S)$.
\begin{dfn}\label{the set R_p}
For a point $p\in S$, we define the \emph{limit set} $R_{p}$ as the set of vectors in $T_{p}M$ that are limits of sequences of the vectors $R_{q}$ defined above at points $q\in M\setminus S$.

\end{dfn}

\begin{figure}[ht]
 \centering
 \includegraphics{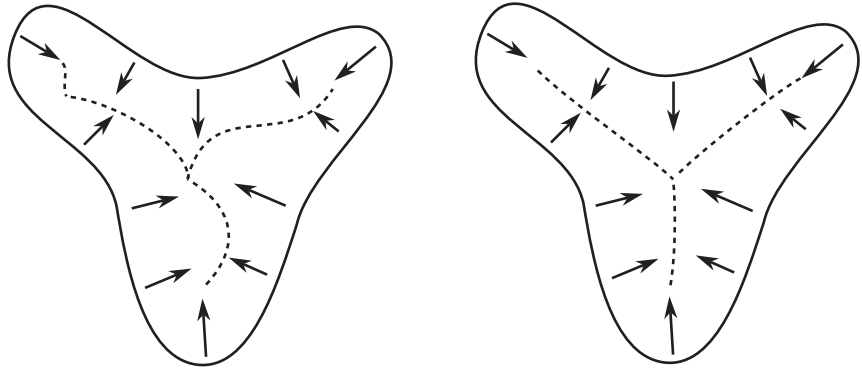}
 \caption{An arbitrary split locus and a balanced split locus}
 \label{fig: One split and one balanced}
\end{figure}

\begin{dfn}\label{split locus}
 A set $S$ that splits $M$ is a split locus iff
$$
S =\overline{
\left\lbrace 
p\in S: \quad \sharp R_{p}\geq 2
\right\rbrace }
$$
\end{dfn}

The role of this condition is to restrict $S$ to its \textit{essential} part. A set that merely splits $M$ could be too big: actually $M$ itself splits $M$.
Finally, we introduce the following more restrictive condition.

\begin{dfn}[Balanced split locus]\label{balanced} 
We say a split locus $S\subset M$ is \emph{balanced} at $p\in S$ iff 
for any sequence $\{p_n\}$ converging to $p$ with $v_{p_{n}}(p)$ and $X_{n}\in R_{p_{n}}$ approaching $v\in T_{x}M$ and  $X_{\infty}\in R_{p}$ respectively,
then
$$
w_{\infty}(v)=\max\left\lbrace w(v)\,:\, w \text{ is dual to some $R\in R_{p}$} \right\rbrace
$$
\noindent where $w_{\infty}$ is the dual of $X_{\infty}$.
We say  $S$ is \emph{balanced} if it is balanced at every point. 

\end{dfn}

\section{Balanced property of the Finsler cut locus}\label{section: Balanced}
In this section we show that the cut locus of a Finsler exponential map is a balanced set.
The proof is the same as in lemma 2.1 in \cite{Itoh Tanaka 2}, only adapted to Finsler manifolds, where angles are not defined.

\begin{prop}\label{cut locus is balanced}
The cut locus of a Finsler manifold $M$ with boundary 
is a balanced split locus.
Moreover, for $p$, $p_{n}$, $v$ and $X_{\infty}$ as in the definition of a balanced split locus, we have
$$
\lim_{n\rightarrow \infty}\dfrac{d(\partial M,p)-d(\partial M,p_{n})}{d(p,p_{n})}=
w_{\infty}(v)
$$
\end{prop}
\begin{proof} 
The cut locus $S$ splits $M$, as follows from the well-known property that if a geodesic $\gamma$ from $\partial M$ to $p=\gamma(t)$ is minimizing, and $s<t$, then $\gamma \vert_{[0,s]}$ is the unique minimizing geodesic from $\partial M$ to $\gamma(s)$, and is non-focal.

It is also a split locus, as follows from the characterization of the cut locus as the closure of the singular set of the function distance to the boundary (as found in \cite{Li Nirenberg} for example). The distance to the boundary is differentiable at a point if and only if there is a unique minimizing geodesic from the point to the boundary.

Next we show that $S$ is balanced. Take any $Y\in R_{p}$, and let $\gamma$ be the minimizing geodesic segment joining $\partial M$ to $p$ with speed $Y$ at $p$. Take any point $q\in\gamma$ that lies in a convex neighborhood of $p$ and use the triangle inequality to get:
$$
d(\partial M, p)-d(\partial M, p_{n}) \geq d(q, p) - d(q,p_{n})
$$
Then the first variation formula yields, for a constant $C$:
$$
d(q, p) - d(q,p_{n})\geq w(v_{p_{n}}(p))d(p_{n},p)-Cd(p,p_{n})^{2}
$$
and we get:
$$
\liminf_{n\rightarrow \infty}\dfrac{d(\partial M,p)-d(\partial M,p_{n})}{d(p,p_{n})}\geq
w(X)
$$
for any $w$ that is dual to a vector in $R_{p}$.

Then consider $X_{\infty}$, let $\gamma$ be the minimizing geodesic segment joining $\partial M$ to $p$ with speed $X_{\infty}$ at $p$, and let $\gamma_{n}$ be the minimizing geodesic segment joining $\partial M$ to $p_{n}$ with speed $X_{n}$ at $p_n$. Take 
points $q_{n}$ in $\gamma_{n}$ that lie in a fix convex neighborhood of $p$. Again:
$$
d(\partial M, p)-d(\partial M, p_{n}) \leq d(q_{n}, p) - d(q_{n},p_{n})
$$
while the first variation formula yields, for a constant $C$:
$$
d(q_{n}, p) - d(q_{n},p_{n})\leq w(v_{p_{n}}(p))d(p_{n},p)- Cd(p,p_{n})^{2}
$$
and thus:
$$
\limsup_{n\rightarrow \infty}\dfrac{d(\partial M,p)-d(\partial M,p_{n})}{d(p,p_{n})}\leq
w_{\infty}(X)
$$

This proves the claim that $S$ is balanced. 
\end{proof}

\section{Focal points in a balanced split locus}\label{section: First}

In this section we prove Theorem \ref{main theorem 3}.
Throughout this section, $M$, $r$, $V$ and $F$ are as in section \ref{section: setting}
and $S$ is a balanced split locus as defined in \ref{balanced}.

\begin{dfn}\label{A2}
A singular point $x\in V$ of the map $F$ is an \emph{A2} point if $ker(dF_x)$ has dimension $1$ and is transversal to the tangent to the set of focal vectors.
\end{dfn}

\paragraph{Remark.} Warner shows in \cite{Warner} that the set of focal points of order $1$ is a smooth (open) hypersurface inside $V$, and that  for adequate coordinate functions in $V$ and $M$, the exponential has the following normal form around any A2 point,
\begin{eqnarray}\label{normal form for A2} 
(x_{1},x_{2},\dots , x_{m})\longrightarrow
(x_{1}^{2},x_{2}, \dots, x_{m})
\end{eqnarray}

\begin{prop}\label{no A2 in S}
 For any $p\in M$ and $X\in R_{p}$, the vector $X$ is not of the form $dF_x(r)$ for any A2 point $x$.

\end{prop}
\begin{proof}
The proof is by contradiction: let $p\in S$ be such that $R_{p}$ contains an A2 vector $Z$. There is a unique $c\in V$ such that $F(c)=p$ and $dF_{c}(r_{c})=Z$.
By the normal form (\ref{normal form for A2}), we see there is a neighborhood $U$ of $c$ such that no other point in $U$ maps to $p$. 
Furthermore, in a neighborhood $B$ of $p$ the image of the focal vectors is a hypersurface $H$ such that all points at one side (call it $B_{1}$) have two preimages of $F\vert_{U}$, all points at the other side $B_{2}$ of $H$ have no preimages, and points at $H$ have one preimage, whose corresponding vector is A2-focal.
It follows that $Z$ is isolated in $R_{p}$.

We notice there is a sequence of points $p_{n}\rightarrow p$ in $B_{2}$ with vectors $Y_{n}\in R_{p_{n}}$ such that  $Y_{n}\rightarrow Y\neq X$. Thus $R_{a}$ does not reduce to $Z$.

The vector $Z$ is tangent to $H$, so we can find a sequence of points $p_{n}\in B_{2}$ approaching $p$ such that
$$
\lim_{n\rightarrow \infty}v_{p_{n}}(p)=Z
$$
We can find a subsequence $p_{n_{k}}$ of the $p_{n}$ and vectors $X_{k}\in R_{p_{n_{k}}}$ such that $X_{k}$ converges to some $X_{\infty}\in R_{p}$. By the above, $X_{\infty}$ is different from $Z$, but $\hat{Z}(X)<1= \hat{Z}(Z)$ (where $\hat{Z}$ is the dual form to $Z$), so the balanced property is violated.
\end{proof}

The following is the analogous to theorem \ref{main theorem 3} for focal points of order $1$.

\begin{prop}[Focal points of order $1$]\label{theorem: conjugate of order 1}
There is a set $N\subset S$ of Hausdorff dimension $n-2$ such that for all $p\in S\setminus N$ and $x\in V$ such that $F(x)=p$ and $d_x F(r_{x})\in R_{p}$, the linear map $d_{x}F$ is non-singular.
\end{prop}
\begin{proof}
The proof is identical to the proof of lemma 2 in \cite{Itoh Tanaka 98} for a cut locus, but we include it here for completeness.
First of all, at the set of focal vectors of order $k\geq 2$ we can apply directly the Morse-Sard-Federer theorem (see \cite{Federer}) to show that the image of the set of focal cut vectors of order $k\geq 2$ has Hausdorff dimension at most $n-2$.

Let $Q$ be the set of focal vectors of order $1$ (recall it is a smooth hypersurface in $V$).
Let $G$ be the set of focal vectors such that the kernel of $dF$ is tangent to the focal locus. Apply the Morse-Sard-Federer theorem again to the map $F\vert_{Q}$ to show that the image of $G$ has Hausdorff dimension at most $n-2$. Finally, the previous result takes cares of the $A2$ points. 
\end{proof}

We now turn to the main result of this paper: we state and prove Theorem \ref{theorem: conjugate of order 2} which has \ref{main theorem 3} as a direct consequence.
In order to study the map $F$ more comfortably, we define the \emph{special coordinates} in a neighborhood of a focal point $z$ of order $k$.

\paragraph{Special coordinates.}
Let $\mathcal{B}=\lbrace v_{1},\dots,v_{n}\rbrace$ be the basis of $T_{z}V$ indicated in the second part of Proposition \ref{regular exponential map}, and $\mathcal{B}'_{F(z)}$ the corresponding basis at $F(z)\in M$ formed by vectors $d_{z}F(v_{1})$, 
$\widetilde{d^{2}_{z} F(v_{1}\sharp v_{2})},\dots,
\widetilde{d^{2}_{z} F(v_{1}\sharp v_{k+1})}$,
and $d_{z}F(v_{i}), i>k+1$.

Make a linear change of coordinates in a neighborhood of $F(z)$ taking $\mathcal{B}'_{F(z)}$ to the canonical basis.
The coordinate functions $F^{i}(x)-F^{i}(z)$ of $F$ for $i\neq 2,\dots,k+1$ can be extended to a coordinate system near $z$ with the help of $k$ functions having $v_{2},\dots,v_{k+1}$ as their respective gradients at $z$. In this coordinates $F$ looks:
\begin{equation}\label{F near order k in special coords} 
 F(x_{1},x_{2},\dots,x_{k+1},\dots,x_{n})=(x_{1},F_{z}^{2}(x),\dots,F_{z}^{k+1}(x),x_{k+2},\dots,x_{n})
\end{equation} 

\begin{theorem} \label{theorem: conjugate of order 2}
Let $M$, $V$, $F$ and $r$ be as in section \ref{section: setting}.
Let $S$ be a balanced split locus (\ref{balanced}).
The set of focal points of order $2$ in $V$ decomposes as the union of two subsets $Q_{2}^{1}$ and $Q_{2}^{2}$ such that:
\begin{itemize}
\item No vector in $Q_{2}^{1}$ maps under $dF$ to a vector in any of the $R_{a}$.

\item The image under $F$ of $Q_{2}^{2}$ has Hausdorff dimension at most $n-3$.
\end{itemize}
\end{theorem}

\begin{proof} 

Let $z$ be a focal point of order $2$ and take special coordinates at $U_{z}$ near $z$. 
In the special coordinates near $z$, $F$ is written:
\begin{equation}
 \label{exp in local coords} 
F(x_{1},x_{2},x_{3},x_4,\dots,x_{n})=(x_{1},F_{z}^{2}(x),F_{z}^{3}(x),x_{4},\dots,x_{n})
\end{equation}
for some functions $F_{z}^{2} $ and $F_{z}^{3} $, and $x=(x_{1},\dots,x_{n})$ in a neighborhood $U_{z}$ of $z$ with $F(0,\dots,0)=(0,\dots,0)$.

The Jacobian of $F$ is:
$$
JF=\begin{bmatrix}
1 & * & * & 0 & \dots & 0\\[2ex]
0 & \frac{\partial F_{z}^{2}}{\partial x_{2}} & \frac{\partial F_{z}^{3}} {\partial x_{2}}& 0 & \dots& 0\\[2ex]
0 & \frac{\partial F_{z}^{2}}{\partial x_{3}} & \frac{\partial F_{z}^{3}} {\partial x_{3}}& 0 & \dots & 0\\[2ex]
0 & * & * & 1 & \dots & 0\\
\vdots&\vdots&\vdots&\vdots&\ddots&\vdots\\
0&* & *&0&\dots & 1
\end{bmatrix}
$$
A point $x$ is of second order if and only if the $2\times 2$ submatrix for the $x_{2}$ and $x_{3}$ variables vanish:
\begin{equation}\label{minor of vars 2 and 3} 
\begin{bmatrix}
  \frac{\partial F_{z}^{2}}{\partial x_{2}} (x) &
 \frac{\partial F_{z}^{3}} {\partial x_{2}} (x) \\[2ex]
 \frac{\partial F_{z}^{2}}{\partial x_{3}}  (x) &
 \frac{\partial F_{z}^{3}} {\partial x_{3}} (x)
\end{bmatrix}
=0
\end{equation} 

We write:
$$
F_{z}^{2}(x)=x_{1}x_{2}+q(x_{2},x_{3})+T^{2}(x)
$$
$$
F_{z}^{3}(x)=x_{1}x_{3}+r(x_{2},x_{3})+T^{3}(x)
$$
where
$ q(x_{2},x_{3})$ and $r(x_{2},x_{3}) $ are the quadratic terms in $x_{2}$ and $x_{3}$ in a Taylor expansion, and 
$T$ consists of terms of order $\geq 3$ in $x_{2}$ and $x_{3}$, and terms of order $\geq 2$ with at least one $x_{i}, i\neq 2,3$.
x
The nature of the polynomials $q$ and $r$ in the special coordinates at $z$ will determine whether $z$ is in $Q_{2}^{1}$ or in $Q_{2}^{2} $. We have the following possibilities:

\begin{enumerate}
 \item either $q$ or $r$ is a sum of squares of homogeneous linear functions in $x_{2}$ and $x_{3}$ (possibly with a global minus sign).
 \item both $q$ and $r$ are products of distinct linear functionals (equivalently, they are difference of squares). Later on, we will split this class further into three types: 2a, 2b and 2c.
 \item one of $q$ and $r$ is zero, the other is not.
 \item both $q$ and $r$ are zero.
\end{enumerate}

We set $Q_{2}^{1}$ to be the points of type 1 and 2c, and $Q_{2}^{2}$ to be the points of type 2a, 3 and 4. Points of type 2b do not appear under the hypothesis of this theorem.

\paragraph{Type 1.}
The proof is similar to Proposition \ref{no A2 in S}. Assume $z=(0,\dots,0)$ is of type 1. If, say, $q$ is a sum of squares, then in the subspace given by $x_{1}=a$ and $x_{4}=\dots=x_{n}=0$, $x_2$ will reach a minimum value that will be greater than $-Ca^{2}$ for some $C>0$.
We learn there is a sequence $p^{k}=(t^{k},-(C+1) (t^{k})^{2},0,\dots,0)$, for $t^{k}\nearrow 0$, approaching $(0,\dots,0)$  with incoming speed $(1,0,\dots,0)$ and staying in the interior of the complement of $F(U)$ for $k$ large enough. 
Pick up any vectors $V_{k}\in R_{p^{k}}$ converging to some $V_{0}$ (passing to a subsequence if necessary). Then $V_{0}$ is different from $(1,0,\dots,0)\in R_{0}$, and 
$$
\widehat{V_{0}}\left( (1,\dots,0)\right) <\widehat{(1,\dots,0)}\left( (1,\dots,0)\right)=1
$$
violating the balanced condition.

\paragraph{Type 2 and 3.}

We take special coordinates at a fixed $x_{0}$ and assume $q\neq 0$. Before we start, we will change coordinates to simplify the expression of $F$ further. Consider a linear change of coordinates near $x$ that mix only the $x_{2}$ and $x_{3}$ coordinates.
$$
\left( 
\begin{array}{c}
 x'_{2}\\
 x'_{3}
\end{array}
\right) 
=
A\cdot 
\left( 
\begin{array}{c}
 x_{2}\\
 x_{3}
\end{array}
\right)
$$
followed by the linear change of coordinates near $p$ that mix only the $y_{2}$ and $y_{3}$ coordinates with the inverse of the matrix above:
$$
\left( 
\begin{array}{c}
 y'_{2}\\
 y'_{3}
\end{array}
\right) 
=
A^{-1}\cdot 
\left( 
\begin{array}{c}
 y_{2}\\
 y_{3}
\end{array}
\right)
$$
Straightforward but tedious calculations show that there is a matrix $A$ such that the map $F$ has the following expression in the coordinates above:
$$
F(x_{1},x_{2},x_{3},x_4,\dots,x_{n})=(x_{1},x_{1}x_{2}+(x_{2}^{2}-x^{2}_{3}),x_{1}x_{3}+r(x_{2},x_{3}),x_4,\dots,x_{n})+T
$$
In other words, we can assume $q(x_{2},x_{3})=(x_{2}^{2}-x^{2}_{3})$.

Take $x_{i}$, fixed and small, $i>3$. At the origin, $JF$ is a diagonal matrix with zeros in the positions $(2,2)$ and $(3,3)$. 
We recall that $z$ is focal of order $2$ iff the submatrix (\ref{minor of vars 2 and 3}) vanishes. This submatrix is the sum of 
\begin{eqnarray}\label{jacobian of the interesting part}
\begin{bmatrix}
x_{1}+2x_{2} & r_{x_{2}}\\
 -2x_{3} & x_{1}+r_{x_{3}}
\end{bmatrix}
\end{eqnarray} 
and some terms that either have as a factor one of the $x_{i}, i>3$, or are quadratic  in $x_{2}$ and $x_{3}$. 

We want to show that, near points of type 3 and some points of type 2, all focal points of order $2$ are contained in a submanifold of codimension $3$.
The claim will follow if we show that the gradients of the four entries span a $3$-dimensional space at points in $U$. For convenience, write $r(x_{2},x_{3})=\alpha x_{2}^{2}+\beta x_{2}x_{3}+\gamma x_{3}^{2}$. It is sufficient that the matrix with the partial derivatives with respect to $x_{i}$ for $i=1,2,3$ of the four entries have rank $3$:
$$
A=
\begin{bmatrix}
1 & 2 & 0\\
0 & 0 & -2\\
0 & 2\alpha & \beta\\
1 & \beta & 2\gamma
\end{bmatrix}
$$
The claim holds for $x_{i}$ small, $i>3$, unless $\alpha=0$ and $\beta=2$. This covers points of type $3$. We say a point of type 2 has type 2a if the rank of the above matrix is $3$. Otherwise, the polynomial $r$ looks:
$$
r(x_{2},x_{3})=2 x_{2}x_{3}+\gamma x_{3}^{2}=2x_{3}(x_{2}+\frac{\gamma}{2}x_{3})
$$
We say a point of type 2 has type 2b if $r$ has the above form and $-1<\frac{\gamma}{2} <1$. We will show that there are integral curves of $r$ arbitrarily close to the one through $z$ without focal points near $z$, which contradicts property 3 in Proposition \ref{regular exponential map}. 

Take a ray $t\rightarrow \zeta_{x_{3}}(t)$ passing through a point $(0,0,x_{3},0,\dots,0) $.
The determinant of \ref{minor of vars 2 and 3} along the ray is:

$$
\begin{array}{rl}
p(t)&=
 \frac{\partial F_{z}^{2}}{\partial x_{2}} (\zeta(t)) 
 \frac{\partial F_{z}^{3}}{\partial x_{3}}  (\zeta(t)) -
 \frac{\partial F_{z}^{3}} {\partial x_{2}} (\zeta(t)) 
 \frac{\partial F_{z}^{2}} {\partial x_{3}} (\zeta(t))\\
&=t^{2}+t(4x_{2}+2\gamma x_{3})+(4x_{2}^{2}+4\gamma x_{2}x_{3}+4x_{3}^{2})
+R_{3}(x_{3},t)\\
&= (t+2x_{2}+\gamma x_{3})^{2}+(4-\gamma^{2})x_{3}^{2}
+R_{3}(x_{3},t)\\
&\geq c(t^{2}+x_{3}^{2})+R_{3}(x_{3},t)
\end{array}
$$
for a remainder $R_{3}$ of order $3$. Thus there is a $\delta >0$ such that for any $x_{3}\neq 0$ and $\vert t\vert<\delta$, $\vert x_{3}\vert<\delta$, $\zeta_{x_{3}}(t)$ is not a focal point.

We have already dealt with points of type 3, 2a and 2b.
Now we turn to the rest of points of type 2 (type 2c). We have either $\frac{\gamma}{2} \geq 1$ or $\frac{\gamma}{2} \leq -1$.
We notice that $x_{2}^{2}- x_{3}^{2}\leq 0$ iff $\vert x_{2}\vert\leq \vert x_{3}\vert $, but whenever $\vert x_{2}\vert\leq \vert x_{3}\vert $, the sign of $r(x_{2},x_{3})$ is the sign of $\gamma$. Thus the second order part of $F$ maps $U$ into the complement of points with negative second coordinate and whose third coordinate has the opposite sign of $\gamma$.

A similar argument as the one for type 1 points yields a contradiction with the balanced condition. If, for example, $\gamma \geq 2$, none of the following points 
$$
x^{k}=(t^{k},-(C+1) (t^{k})^{2},-(C+1) (t^{k})^{2},..0,)
$$ 
is in $F(U)$, for $t^{k}\rightarrow 0$. But then we can carry a vector other than $(1,0,\dots,0)$ as we approach $F(x_{0})$.


\paragraph{Type 4.}
Let $z$ be a focal point of order $2$.
We show now that the image of the points of type 4 inside $U_{z}$ has Hausdorff dimension at most $n-3$. $U_z$ is an open set around an arbitrary point $z$ of order $2$, and thus the result follows.

First, we find that for any point $x$ of type 4, we have $d^2_x F(v\sharp w)=0$ for all $v,w\in \ker d_x F$, making the computation in the special coordinates at $x\in U_z$ (see section \ref{subsection: regular exponential map} for the definition of $d^2 F$).

Then we switch to the special coordinates around $z$. 
In these coordinates, the kernel of $dF$ at $x$ is generated by $\frac{\partial }{\partial x_{2}} $ and $\frac{\partial }{\partial x_{3}} $.
Thus $\frac{\partial^2 F_{z}^{2}}{\partial x_{i}x_{j}} =0$ for $i,j=2,3$ at any point $x\in U_z$ of type 4.

The set of focal points of order $2$ is contained in the set $H=\lbrace\frac{\partial F_{z}^{2}}{\partial x_{2}} (x) =0 \rbrace$. 
This set is a smooth hypersurface: the second property in \ref{regular exponential map} implies that $\frac{\partial^2 F_{z}^{2}}{\partial x_{1}x_{2}} \neq 0 $ at points of $H$.
At every focal point of type 4, the kernel of $dF$ is contained in the tangent to $H$. Thus focal points of type 4 are focal points of the restriction of $F$ to $H$. 
The Morse-Sard-Federer theorem applies, and the image of the set of points of type 4 has Hausdorff dimension $n-3$.

% ...
%
%

\end{proof}

\begin{proof}[Proof of Theorem \ref{main theorem 3}]
Follows immediately from the above, setting $N=F(Q^{2}_{2})$.
\end{proof}

\section{Structure up to codimension 3}\label{section: Second}

This section contains the proof of \ref{complete description}, splitted into several lemmas. All of them are known for cut loci in riemannian manifolds, but we repeat the proof so that it applies to balanced split loci in Finsler manifolds.

\begin{dfn}\label{cleave point}
We say $p\in S$ is a \emph{cleave} point iff $R_{p}$ has two elements $X^{1}$ and $X^{2}$, with $(p,X^{1})=(F(y_{1}),dF_{y_{1}}(r_{y_{1}}))$ and $(p,X^{2})=(F(y_{2}),dF_{y_{2}}(r_{y_{2}}))$, and both $dF_{y_{1}}$ and $dF_{y_{2}}$ are non-singular.
\end{dfn}

\begin{prop}\label{main theorem 2}
$\mathcal{C}$ is a $(n-1)$-dimensional manifold.
\end{prop}

\begin{proof}
Let $p=F(y_{1})=F(y_{2})$ be a cleave point, with $R_{p}=\{dF_{y_{1}}(r), dF_{y_{2}}(r)\}$.
We can find a small neighborhood $U$ of $p$ so that the following conditions are satisfied:
\begin{enumerate}
 \item $U$ is the diffeomorphic image of neighborhoods $U_{1}$ and $U_{2}$ of the points $y_{1}$ and $y_{2}$.

Thus, the two smooth vector fields $X^{1}_{q}=dF\vert_{U_{1}}(r)$ and $X^{2}_{q}=dF\vert_{U_{2}}(r)$ are defined in points $q\in U$.
 \item At all points $q\in U$, $R_{q}\subset \lbrace X^{1}_{q},X^{2}_{q}\rbrace$. Other vectors must be images of the vector $r$ at points not in $U_{1}$ or $U_{2}$, and if they accumulate near $p$ there is a subsequence converging to a vector that is neither $X_{1}$ nor $X_{2}$.

 \item Let $\Gamma_{1}$ be an hypersurface in $U_{1}$ passing through $y_{1}$ and transversal to $X_{1}$, and let $\tilde{\Gamma} =F(\Gamma)$. We define local coordinates $q=(x, t)$ in $U$, where $x\in \tilde{\Gamma}$ and $t\in \RR$ are the unique values for which $q$ is obtained by following the integral curve of $X^{1}$ that starts at $x$ for time $t$.
$U$ is a cube in these coordinates.

\end{enumerate}

We will show that $S$ is a  graph in the coordinates $(x,t)$.
Let $A_{i}$ be the set of points $q$ for which $R_{q}$ contains $X^{i}_{q}$, for $i=1,2$. 
By the hypothesis, $S=A_{1}\cap A_{2}$.

Every tangent vector $v$ to $S$ at $q\in S$ (in the sense of \ref{approximate tangent cone}), satisfies the following property (where $\hat{X}$ is the dual covector to a vector $X\in TM$.):
$$
\hat{X}^{i}(v)=
\max_{Y\in R_{p}}
\hat{Y}(v)
$$
which in this case amounts to $\hat{X}^{1}(v)=\hat{X}^{2}(v)$, or
$$
v\in \ker(\hat{X}^{1}-\hat{X}^{2})
$$

We can define in $U$ the smooth distribution $D=\ker(\hat{X}^{1}-\hat{X}^{2})$. 
$S$ is a closed set whose approximate tangent space is contained in $D$.

We first claim that for all $x$, there is at most one time $t_{0}$ such that $(x,t_{0})$ is in $S$. 
If $(x,t)$ is in $A_{1}$, $R_{(x,t)}$ contains $X^{1}$ and, unless $(x,s)$ is contained in $A_{1}$ for $s$ in an interval $(t-\varepsilon,t)$, we can find a sequence $(x_{n},t_{n})$ converging to $(x,t)$ with $t_{n}\nearrow t$ and carrying vectors $X^{2}$. The incoming vector is $X^{1}$, but
$$
\tilde{X^2}(X^1)<\tilde{X^1}(X^1)=1
$$
which contradicts the balanced property.
Analogously, if $R_{(x,t)}$ contains $X^{2}$ there is an interval $(t,t+\varepsilon)$ such that $(x,s)$ is contained in $A_{2}$ for all $s$ in the interval. Otherwise there is a sequence $(x_{n},t_{n})$ converging to $(x,t)$ with $t_{n}\searrow t$ and carrying vectors $X^{1}$. The incoming vector is $-X^{1}$, but
$$
-1=\tilde{X^1}(-X^1)<\tilde{X^2}(-X^1)
$$
which is again a contradiction. The claim follows easily.

We show next that the set of $x$ for which there is a $t$ with $(x,t)\in S$ is open and closed in $\Gamma$, and thus $S$ is the graph of a function $h$ over $\Gamma$.
Take $(x,t)\in U\cap S$ 
and choose a cone $D_{\varepsilon}$ around $D_{x}$. We can assume the cone intersects $\partial U$ only in the $x$ boundary. There must be a point in $S$ of the form $(x',t')$ inside the cone for all $x'$ sufficiently close to $x$: otherwise there is either a sequence $(x_n,t_n)$ approaching $(x,t)$ with $t_n>h_+(x)$ ($h$ being the upper graph of the cone $D_{\varepsilon}$) and carrying vectors $X^{1}$ or a similar sequence with $t_n<h_-(x)$ and carrying vectors $X^{2}$.
Both options violate the balanced condition. 
Closedness follows trivially from the definition of $S$.

Define $t=h(x)$ whenever $(x,t)\in S$. The tangent to the graph of $h$ is given by $D$ at every point, thus $S$ is smooth and indeed an integral maximal submanifold of $D$.
\end{proof}

\paragraph{Remark.} It follows from the proof above that there cannot be any balanced split locus unless $D$ is integrable. This is not strange, as the sister notion of cut locus does not make sense if $D$ is not integrable.

We recall that the orthogonal distribution to a geodesic vector field is parallel for that vector field, so the distribution is integrable at one point of the geodesic if and only if it is integrable at any other point. In particular, if the vector field leaves a hypersurface orthogonally (which is the case for a cut locus) the distribution $D$ (which is the difference of the orthogonal distributions to two geodesic vector fields) is integrable. It also follows from \ref{Hamilton Jacobi} that the characteristic vector field in a Hamilton-Jacobi problem has an integrable orthogonal distribution.

\paragraph{Remark.} We commented earlier on our intention of studying whether a balanced split locus is actually a cut locus. The proof of the above lemma showed there is a unique sheet of cleave points near a given point in a balanced split loci. It is not too hard to deal with the case when all incoming geodesics are non-focal, but focal geodesics pose a major problem.

\begin{prop}\label{main theorem 4}
The set of points $p\in S$ where $co\, (R^{\ast}_{p})$ has dimension $k$ is $(n-k)$-rectifiable.
\end{prop}
\begin{proof}
Throughout the proof, let $\hat{X}$ be the dual covector to the vector $X\in TM$.

Let $p_{n}$ be a sequence of points such that $co\,(R^{\ast}_{p_{n}})$ contains a $k$-dimensional ball of radius greater than $\delta$. Suppose they converge to a point $p$ and $v_{p_{n}}(p)$ converges to a vector $\eta$.

We take a neighborhood $U$ of $p$ and fix product coordinates in $\pi^{-1}(U)$ of the form $U\times \Rn$.
Then, we extract a subsequence of $p_{n}$ and vectors $X_{n}^{1}\in R_{p_{n}}$ such that $X_{n}^{1}$ converge to a vector $X^{1}$ in $R_{p}$. Outside a ball of radius $c\delta$ at $\hat{X}_{n}^{1}$, where $c$ is a fixed constant and $n>>0$, there must be vectors in $R_{p_{n}}$, and we can extract a subsequence of $p_{n}$ and vectors $X_{n}^{2}$ converging to a vector $X^{2}$ such that $\hat{X}^{2}$ is at a distance at least $c\delta$ of $\hat{X}^{1}$. Iteration of this process yields a converging sequence $p_{n}$ and $k$ vectors 
$$
X_{n}^{1}, .. ,X_{n}^{k}\in R_{p_{n}}
$$
converging to vectors
$$
X^{1}, .. ,X^{k}\in R_{p}
$$
such that the distance between $\hat{X}^{k}$ and the linear span of $\hat{X}^{1},..\hat{X}^{k-1}$ is at least $c\delta$, so that $coV^{\ast}_{p}$ contains a $k$-dimensional ball of radius at least $c'\delta$.

The balanced property implies that the $\hat{X}^{j}$ evaluate to the same value at $\eta$, which is also the maximum value of the $\hat{Z}(\eta)$ for a vector $Z$ in $R_{p}$. In other words, the convex hull of the $\hat{X}^{j}$ belong to the face of $R^{\ast}_{p}$ that is exposed by $\eta$. If $co\,R_{p}^{\ast}$ is $k$-dimensional, $\eta$ belongs to 
$$
\begin{array}{rl}
 \left( co\,R_{p}^{\ast}\right)^{\perp} =&
\left\lbrace 
    v\in T_{p} M \::\quad  \langle w,v\rangle \text{ is constant for }w\in co\,R_{p}^{\ast}
\right\rbrace \\[2ex]
=&
\left\lbrace 
    v\in T_{p} M \::\quad  \langle \hat{X},v\rangle \text{ is constant for }X\in R_{p}
\right\rbrace
\end{array} 
$$
which is a $n-k$ dimensional subspace.

Let $\Sigma^{k}_{\delta}$ be the set of points 
$p\in S$ for which $co\,R_{p}^{\ast}$ is $k$-dimensional and contains a $k$-dimensional ball of radius greater than or equal to $\delta$.
We have shown that all tangent directions to $\Sigma^{k}_{\delta}$ at a point $p$ are contained in a $n-k$ dimensional subspace. We can apply theorem 3.1 in \cite{Alberti Ambrosio Cannarsa} to deduce $\Sigma^{k}_{\delta}$ is $n-k$ rectifiable, so their union for all $\delta>0$ is rectifiable too.

\end{proof}

\section{Appendix: Finsler geometry and exponential maps}
\begin{dfn}\label{dual one form}
The \emph{dual one form} to a vector $V\in T_{p}M$ with respect to a Finsler metric $\varphi $ is the unique one form $\omega \in T^{\ast}_{p}M$ such that $\omega(V)=\varphi(V)^{2}$ and $\omega\vert_{H}=0$, where $H$ is the hyperplane tangent to the level set 
$$
\left\lbrace W \in T_{p}M: \varphi(W) = \varphi(V) \right\rbrace 
$$
at $V$. It coincides with the usual definition of dual one form in Riemannian geometry.

For a vector field, the dual differential one-form is obtained by applying the above construction at every point.
\end{dfn}

\paragraph{Remark.} In coordinates, the dual one form $w$ to the vector $V$ is given by:
$$
w_{j}=\frac{\partial\varphi}{\partial V^{j}}(p,V)
$$

Actually $\varphi$ is 1-homogeneous, so Euler's identity yields:
$$
w_{j}V^{j}=\frac{\partial\varphi}{\partial V^{j}}(p,V)V^{j}=1
$$
and, for a curve $\gamma (-\varepsilon, \varepsilon)\rightarrow T_{p}M$ such that $\gamma(0)=V$, $\varphi(\gamma(t))=\varphi(V)$ and $\gamma'(0)=z$,
$$
w_{j}z^{j}=\frac{\partial}{\partial t}\vert_{t=0}\varphi (\gamma(t))=0
$$

\paragraph{Remark.} The hypothesis on $H$ imply that the orthogonal form to a vector is unique.

\begin{dfn}
The orthogonal hyperplane to a vector is the kernel of its dual one form.
The orthogonal distribution to a vector field is defined pointwise.

There are two unit vectors with a given hyperplane as orthogonal hyperplane. The first need not to be the opposite of the second unless $H$ is symmetric ($H(-v)=H(v)$).
We can define two unit normal vectors to a hypersurface (the \emph{inner} normal and \emph{outer} normal).
\end{dfn}
\subsection{Regular exponential map}\label{subsection: regular exponential map}
The following proposition states some properties of a Finsler exponential map that correspond approximately to the definition of \emph{regular exponential map} introduced in \cite{Warner}:

\begin{prop}\label{regular exponential map}
In the setting \ref{section: setting}
the following holds:
\begin{itemize}
 \item $dF_{x}(r_{x})$ is a non zero vector in $T_{F(x)}M$.
 \item 
at every point $x\in V$ there is a basis 
$$
\mathcal{B}=\left\lbrace v_{1},..,v_{n} \right\rbrace 
$$
of $T_{x}V$ where $r=v_{1}$ and $v_{2},..,v_{k}$ span $\ker dF_{x}$, and such that:
$$\mathcal{B}'=
\left\lbrace 
dF(v_{1}),
\widetilde{d^{2}F(r\sharp v_{2})},\dots 
\widetilde{d^{2}F(r\sharp v_{k})},
dF(v_{k+1}),\dots dF(v_{n})
\right\rbrace 
$$
is a basis of $T_{F(x)}M$, where $\widetilde{d^{2}F(r\sharp v_{2})}$ is a representative of $d^{2}F(r\sharp v_{2})\in T_{F(x)}M/dF(TV_{x}) $.

 \item Any point $x\in V$ has a neighborhood $U$ such that for any ray $\gamma $ (an integral curve of $r$), the sum of the dimensions of the kernels of $dF$ at points in $\gamma \cap U$ is constant.
 \item For any two points $x_{1}\neq x_{2}$ in $V$ with $F(x_{1})=F(x_{2})$, $dF_{x_{1}}(r_{x_{1}})\neq dF_{x_{2}}(r_{x_{2}})$
\end{itemize}
\end{prop}
\paragraph{Proof.} The first three properties follow from the work of Warner \cite[Theorem 4.5]{Warner} for a Finsler exponential map. We emphasize that they are local properties. The last one follows from the uniqueness property for second order ODEs. We remark that the second property implies the last one locally. Combined, they imply that the map $p\rightarrow (F(p),dF_{p}r) $ is an embedding of $V$ into $TM$.

Indeed, properties 1 and 3 are found in standard textbooks (\cite{Milnor}). For the convenience of the reader, we recall some of the notation in \cite{Warner} and show the equivalence of the second property with his condition (R2) on page 577.

\begin{itemize}
 \item A second order tangent vector at $x\in \Rn$ is a map $\sigma$:
$$
\sigma(f)=\sum_{i,j}a_{ij}\frac{\partial^{2} f}{\partial x_{i}\partial x_{j}}
$$

 \item The second order differential of $F:V\rightarrow M$ at $x$ is a map $d^{2}_{x}F:T^{2}_{x}V \rightarrow T^{2}_{x}M $ defined by:
$$
d^{2}_{x}F(\sigma)f=\sigma(f\circ F)
$$
 \item The symmetric product $v\sharp w$ of $v\in T_{x}V$ and $w\in T_{x}V$ is a well defined element of $T^{2}_{x}V/T_{x}V$ 
 with a representative given by the formula:
$$
(v\sharp w)f=\frac{1}{2}(v( w(f))+w (v(f)))
$$
for arbitrary extensions of $v$ to $w$ to vector fields near $x$.
\item The map $d^{2}_{x}F $ induces the map $d^{2}F:T^{2}V_{x}/TV_{x} \rightarrow T^{2}_{F(x)}M/dF(TV_{x})$ by the standard procedure in linear algebra.
\item For $x\in V$, $v\in T_{x}V$ and $w\in \ker dF_{x}$, $d^{2}F(v\sharp w)$ makes sense as a vector in the space $T_{F(x)}M/dF(TV_{x}) $. For any extension of $v$ and $w$, the vector $d^{2}F(v\sharp w)$ is a first order vector.
\end{itemize}

Thus, our condition is equivalent to property (R2) of Warner:

At any point $x$ where $\ker dF_{x}\neq 0$, the map $d^{2}F:T^{2}V_{x}/TV_{x} \rightarrow T^{2}M_{F(x)}/dF(TV_{x})$ sends $\langle r_{x}\rangle \sharp \ker dF_{x}$ isomorphically onto  $T_{F(x)}M/dF(TV_{x})$.

Finally, we recall that $d_{x}F(v)$ is the Jacobi field of the variation $F(\phi_{t}(x+sv))$ at $F(x)$, where $\phi_{t}$ is the flow of $r$ and, whenever $d_{x}F(v)=0$, $d^{2}_{x}F(v\sharp r)$ is represented by the derivative of the Jacobi field along the geodesic.

\end{document}